\documentclass{amsart}
\usepackage{verbatim}
\usepackage{amsmath} 
\usepackage{amssymb}
\usepackage[active]{srcltx}
\usepackage{enumerate}

\newtheorem{theorem}{Theorem}[section]
\newtheorem{proposition}[theorem]{Proposition}
\newtheorem{lemma}[theorem]{Lemma}
\newtheorem{corollary}[theorem]{Corollary}
\newtheorem{fact}[theorem]{Fact}
\newtheorem{claim}[theorem]{Claim}
\newtheorem{observation}[theorem]{Observation}
\newtheorem{problem}[theorem]{Problem}

\theoremstyle{definition}

\newtheorem{definition}[theorem]{Definition}

\newcommand{\add}{\operatorname{add}}
\newcommand{\pcf}{\operatorname{pcf}}

\newcommand{\nad}{\operatorname{{ADD}}}

 \newcommand{\Bd}{\operatorname{Bd}}
 
 \newcommand{\Ke}{\operatorname{In}}

\def\<{\left\langle}
\def\>{\right\rangle}
\def\cf{\operatorname{cf}}
\newcommand{\restr}{\restriction}

\newcommand{\Int}{{2^\omega}}

\newcommand{\setm}{\setminus}
\newcommand{\empt}{\emptyset}
\newcommand{\subs}{\subset}
\def\br#1;#2;{\bigl[ {#1} \bigr]^ {#2} }

\newcommand{\mc}[1]{\mathcal{#1}}

\newcommand{\mf}[1]{\mathfrak{#1}}

\author[L. Soukup]{Lajos Soukup}
\thanks{The preparation of this paper was partially
supported by  Bolyai Grant, OTKA grants K 61600 and K 68262.}
\address
      { Alfr{\'e}d R{\'e}nyi Institute of Mathematics, Hungarian Academy of Sciences, Budapest, Hungary  }
\email{soukup@renyi.hu}
\urladdr{http://www.renyi.hu/$\sim  $soukup}
\title%[Pcf theory]
   {Pcf theory and  cardinal invariants of the reals}
\begin{document}
\begin{abstract}
The {\em additivity spectrum} $\nad(\mc{I})$ of an   ideal $\mc{I}\subset \mc{P}(I)$   
is the set of all regular cardinals $\kappa$ such that there is an increasing 
chain $\{A_\alpha:\alpha<\kappa\}\subs \mc{I} $  with $\cup_{\alpha<\kappa}A_\alpha\notin \mc{I}$.

We investigate which set $A$ of regular cardinals can be the additivity spectrum of 
certain ideals.

Assume that  $\mc{I}=\mc{B}$ or $\mc{I}=\mc{N}$, where
$\mc{B}$ denotes the  ${\sigma}$-ideal  
generated by the compact subsets of the Baire 
space $\omega^\omega$, and $\mc{N}$ is the ideal of the null sets.

We show that 
if $A$ 
 is a non-empty progressive set of uncountable regular cardinals
and  $\pcf(A)=A$ then 
$\nad(\mc{I})=A$ in some c.c.c generic extension of the ground model.
On the other hand, we also show that 
 if $A$ is a countable subset of $\nad(\mc{I})$ then 
$\pcf(A)\subs \nad(\mc{I})$.

For countable sets these results  give a full characterization of the additivity spectrum of $\mc{I}$:
a non-empty countable set $A$ of uncountable regular cardinals
can be  $\nad(\mc{I})$ in some c.c.c generic extension iff
$A=\pcf(A)$. 
 \end{abstract}

\maketitle

\section{Introduction}
Many cardinal invariants   
are defined in the following way:
we consider a family  $\mf{X}\subs \mc P(\br {\omega};{\omega};)$
and define our cardinal invariant $\mf x$ as 
$\mf{x}=\min\{|X|:X\in \mf{X}\}$
or
$\mf{x}=\sup\{|X|:X\in \mf{X}\}$.
The set  $\{|X|:X\in \mf{X}\}$ is called the     
{\em  spectrum of  $\mf x$}.

For example, consider the family  
$\mf A=\{\mc A\subs \br \omega;\omega;:\text{$\mc A$ is a MAD}\}$. Then   
$\mf a=\min\{|\mc A|:\mc A\in \mf A\}$, so the we can say  that 
the spectrum of $\mf a$ is the cardinalities of the maximal almost disjoint subfamilies of 
$\br \omega;\omega;$.

The value of many cardinal invariants  can be modified  almost freely by 
using a  suitable forcing, but their spectrums  
should satisfy more requirements.

In \cite{ST} Shelah and Thomas investigated the connections between the cofinality spectrum of certain groups and pcf theory.
Denote ${\rm CF}(Sym(\omega))$  the cofinality
spectrum  of the group of all permutation of natural
numbers, i.e. the set of regular cardinals
$\lambda$ such that $Sym(\omega)$ is the union of an increasing chain of $\lambda$ proper
subgroups. 
Shelah and Thomas  showed that 
$CF(Sym(\omega))$ cannot be an arbitrarily prescribed set of regular uncountable cardinals:
if $A=\< \lambda_n: n\in\omega \>$ is a strictly increasing
sequence of elements of ${\rm CF}(Sym(\omega))$, then ${\rm pcf}(A)\subseteq {\rm
  CF}(Sym(\omega))$.  On the other hand, they also showed that 
if $K$ is a set of regular cardinals which satisfies certain natural requirements
(see \cite[Theorem 1.3]{ST}) then
$CF(Sym(\omega))=K$ in a certain c.c.c generic extension.

In this paper we  investigate  the {\em additivity spectrum} 
of certain ideals in a similar style.
 Denote $\mf{Reg}$ the class of all infinite regular cardinals.
Given any ideal $\mc{I}\subset \mc{P}(I)$ for each $A\in  \mc{I}^+$ 
put
\begin{displaymath}
\nad(\mc{I},A)=\{\kappa\in \mf{Reg}: 
\exists \text{ increasing } \{A_\alpha:\alpha<\kappa\}\subs \mc{I} \text{  s.t.  }
 \cup_{\alpha<\kappa}A_\alpha=A\},
\end{displaymath}
and let
\begin{displaymath}
\nad(\mc{I})=\cup\{\nad(\mc{I},A):A\in \mc{I}^+\}.
\end{displaymath}
Clearly $\add(\mc{I})=\min \nad(\mc{I})$. 
We will say that $\nad(\mc{I})$ is the {\em additivity spectrum of $\mc{I}$}.

As usual, $\mc M$ and $\mc N$ denote the  null and the meager ideals, respectively.
Let $\mc{B}$ denote  the ${\sigma}$-ideal  generated by 
the compact  subsets of ${\omega}^{\omega}$.
 We have
\begin{displaymath}
 \mc B=\{F\subs \br \omega;\omega; :\text{$F$ is $\le^*$-bounded }\}.
\end{displaymath}
So the poset $\<\omega^\omega, \le^*\>$  has a natural, 
cofinal,  order preserving embedding  
$\Phi$ into $\<\mc B,\subs\>$ 
defined by the formula  $\Phi(b)=\{x:x\le^* b\}$.  
Denote by  $\nad(\<\omega^\omega,\le^*\>)$   %of  $\<\omega^\omega,\le^*\>$ 
 the set of all regular cardinals $\kappa$ such that there is an unbounded $\le^*$-increasing 
chain $\{b_\alpha:\alpha<\kappa\}\subs \omega^\omega$.
Clearly $\nad(\mc B)\supseteq \nad(\<\omega^\omega,\le^*\>)$ and
$\mf b=\min \nad(\mc B)=\min \nad(\<\omega^\omega,\le^*\>)$.
Farah, \cite{Fa},  proved that if GCH holds in the ground model then given any non-empty set 
$A$ of uncountable regular 
cardinals with $\aleph_1\in A$ we have $\nad(\<\omega^\omega,\le^*\>=A$
in some c.c.c extension  of the ground model.
%(see section  \ref{sc:gen_pos} for more explanation). 
So $\nad(\<\omega^\omega,\le^*\>)$ does not have  any 
closedness property.
Moreover,
standard forcing arguments show that 
 $\nad(\mc{I})\cap \{\aleph_n:1\le n<\omega\}$ can also be arbitrary, 
where $\mc{I}\in \{\mc{B},\mc{M},\mc{B}\}$.

However, the situation change dramatically if we consider the whole spectrum 
$\nad(\mc{I})$.
On one hand, we show that if $\mc{I}\in\{\mc{B},\mc{N}\}$ then  $\nad(\mc{I})$ should be closed under certain
pcf operations:  if $A$ is a countable subset of $\nad(\mc{I})$ then 
$\pcf(A)\subs \nad(\mc{I})$ (see Theorems \ref{tm:n_closed} and \ref{tm:b_closed}).

On the other hand, 
we show that if $A$ is a non-empty set of uncountable regular cardinals, $|A|<\min(A)^{+n}$ for some $n\in\omega$ 
(especially if $A$ is progressive), and  $\pcf(A)=A$ then 
$\nad(\mc{I})=A$ in some c.c.c generic extension of the ground model (see Theorem \ref{tm:spec_sharp_hechler}).

For countable sets these results  give a full characterization of the additivity spectrum of $\mc{I}$:
a non-empty countable set $A$ of uncountable regular cardinals
can be  $\nad(\mc{I})$ in some c.c.c generic extension iff
$A=\pcf(A)$.

% For the meager ideal $\mc M$ we have  a much weaker result:
% if $ A\subs  \pcf( \nad(\mc{M},X)%\cap \{\aleph_n:1\le n<\omega\}
% ) $ is countable then
% $\pcf(A)\subs  \nad(\mc{M},X)$ as well.

\section{Construction of additivity spectrums}\label{sc:neg}

To start with we recall some results from pcf-theory.
We will use the notation and terminology of \cite{AM}.
A set $A\subs \mf{Reg}$ is {\em progressive} iff $|A|<\min(A)$. 

The proofs of the next two proposition are standard applications of pcf theory, and they should have been well-known, but the
author was unable to find reference. Proposition \ref{pr:bound} is similar to
\cite[Theorem 3.20]{ST}, but
we do not use   assumption concerning the cardinal arithmetic.

\begin{proposition}\label{pr:succ}
Assume that $A=\pcf(A)\subs \mf{Reg}$ is a progressive set,  and $\lambda\in \mf{Reg}$.
Then there is a family $\mc F\subs \prod A$  with $|\mc F|<\lambda$ such that
for each $g,h\in \prod A$ 
\begin{displaymath}
\text{if $g<_{J_{<\lambda}[A]}h$ then there is $f\in \mc F$ such that 
$g<\max(f,h)$.
} 
\end{displaymath}
\end{proposition}

\begin{proof}
For each $\mu\in \pcf(A)=A$ let $B_\mu\subs A$ be a generator of 
$J_{<\mu^+}[A]$, i.e. 
\begin{displaymath}
J_{<\mu^+}[A]=\<J_{<\mu}[A]\cup\{B_\mu\}\>_{gen}.
\end{displaymath}
Since $cf(\<\prod B_\lambda,\le\>)=\max \pcf(B_\mu)=\mu$  by \cite[Theorem 4.4]{AM}, we can fix 
a family $\mc F_\mu\subs \prod B_\mu$ with $|B_\mu|=\mu$ such that 
$\mc F_\mu$ is cofinal in $\<\prod B_\lambda,\le\>$.

We claim that 
\begin{displaymath}
\mc F=\{\max(f^{\mu_1}_1,\dots, f^{\mu_n}_n): \mu_1<\dots,\mu_n<\lambda,f^{\mu_i}_i\in \mc F_{\mu_i}\}
\end{displaymath}
satisfies the requirements.

Since $A$ is progressive, $|\mc F|\le \sup (A\cap \lambda)<\lambda$.

Assume that $g<_{J_{<\lambda}[A]}h$ for some $g,h\in \prod A$.
Let $X=\{a\in A: g(a)\ge h(a)\}$. Then $X\in J_{<\lambda}[A]$, so there are
$\mu_1,\dots \mu_n\in \pcf(A)\cap \lambda=A\cap \lambda$ such that 
$X\subs B_{\mu_1}\cup \dots \cup B_{\mu_n}$.
For each $1\le i\le n$ choose $f^{\mu_i}_i\in \mc F_{\mu_i}$ with
$g\restriction B_{\mu_i}< f^{\mu_i}_i$.

Then $g<\max (h,f^{\mu_1}_1,\dots, f^{\mu_n}_n )$ and
$\max (f^{\mu_1}_1,\dots, f^{\mu_n}_n )\in \mc F$.
\end{proof}

\begin{proposition}\label{pr:bound}
Assume that $A=\pcf(A)\subs \mf{Reg}$ is a progressive set,  and $\lambda\in \mf{Reg}\setm A$.
If $\<g_\alpha:\alpha<\lambda\>\subs \prod A$ then there are 
$K\in \br \lambda;\lambda;$ and $s\in \prod A$ such that
$g_\alpha<s$ for each $\alpha\in K$. 
\end{proposition}

\begin{proof}
If $\lambda >\max \pcf(A)$ then the equality $cf\<\prod A,<\>=\max \pcf(A)$ yields the result.
So we can assume $\lambda<\max \pcf(A)$.

Since $\lambda\notin \pcf(A)$ we have 
$ J_{<\lambda}[A]=  J_{<\lambda^+}[A]$.
So the poset $\< \prod A, <_{J_{<\lambda}[A]}  \>$ is 
$\lambda^+$-directed. Thus there is $h\in \prod A$ such that 
$g_\alpha <_{J_{<\lambda}[A]} h  $ for each $\alpha<\lambda$.

By proposition \ref{pr:succ} there is a family $\mc F\subs  \prod A$ with $|\mc F|<\lambda$
such that for each $\alpha<\lambda$ there is $f_\alpha\in \mc F$ such that 
$g_\alpha<\max (h, f_\alpha)$.
Since $|\mc F|<\lambda$ there are $K\in \br \lambda;\lambda;$ and $f\in \mc F$
such that $f_\alpha=f$ for each $\alpha\in K$.

Then $s=\max(h,f)\in \prod A$ and $K\in \br \lambda;\lambda;$ satisfy  the requirements. 
\end{proof}

We also need the following observation which is a trivial version 
of Proposition \ref{pr:bound} for finite sets.

\begin{observation}\label{ob:finite}
Assume that  $F\subs \mf{Reg}$ is a finite set  and  $\lambda\in \mf{Reg}\setm F$.
If $\<g_\alpha:\alpha<\lambda\>\subs \prod F$ then there are 
$K\in \br \lambda;\lambda;$ and $s\in \prod F$ such that
$g_\alpha\le s$ for each $\alpha\in K$. 
\end{observation}

\begin{proof}
Let $F_1=F\cap \lambda$ and $F_2=F\setm \kappa$.
Since $|\prod F_1|<\lambda=\cf(\lambda)$ there are
$K\in \br \lambda;\lambda;$ and $s_1\in \prod F_1$ such that
$g_\alpha\restriction F_1=s_1$ for each $\alpha\in K$.

Now define $s_2\in \prod F_2$ as follows:  $s_2(a)=\sup\{g_\alpha(a):\alpha\in K\}$.
Then $K$ and $s=s_1^\frown s_2$ satisfy the requirements.   
\end{proof}

\begin{theorem}
\label{tm:spec_sharp_hechler}
Assume that  $\mc{I}$ is one of the ideals $\mc{B},\mc{M}$ and $\mc{N}$.
If $A=\pcf(A)$ is a non-empty set of uncountable regular cardinals, $|A|<\min(A)^{+n}$ for some $n\in\omega$,
then $A=\nad(\mc{I})$ in some c.c.c generic extension $V^P$.

Especially, if $\empt\ne Y\subs \pcf(\{\aleph_n:1 \le n<\omega\})$ then 
$\pcf(Y)=\nad(\mc{I})$ in some c.c.c generic extension $V^P$.
\end{theorem}

The proof is based on Theorem \ref{tm:sharp_hechler} below. To formulate it we need the following definition.

\begin{definition}
Let  $\varphi$ be  a formula with one free variable, 
and assume that   
$ZFC\vdash $ ''{\em $I_\varphi=\{x:\varphi(x)\}$ is an ideal}''.
We say that the ideal $\mc{I}_\varphi$ has the {\em Hechler property} iff  
given any   ${\sigma}$-directed poset $Q$ there is a c.c.c poset $P$
 such that 
\begin{displaymath}
V^P\models \text{some cofinal subset $\{I_q:q\in Q\}$ of $\<\mc{I}, \subs\>$ is order isomorphic to $Q$.}
\end{displaymath}

\end{definition}

If $\mc{I}_\varphi=\mc{I}_\psi$, then clearly $\mc{I}_\varphi$ is Hechler iff $\mc{I}_\psi$ is.
So for well-known ideals, i.e.  for $\mc{B}$ and $ \mc{N}$, we will speak about
the {\em Hechler property of $\mc{I}$} instead of the Hechler property of $\mc{I}_\phi$, where
$\phi$ is one of the many  equivalent definitions of $\mc{I}$.

\begin{theorem}\label{tm:sharp_hechler}
Assume that the ideal $\mc{I}$ has the Hechler property.
If $A=\pcf(A)$ is a non-empty set of uncountable regular cardinals, $|A|<\min(A)^{+n}$ for some $n\in\omega$,
then in some c.c.c generic extension $V^P$
we have $A=\nad(\mc{I})$.
\end{theorem}

\begin{proof}[Proof of theorem \ref{tm:spec_sharp_hechler} from Theorem \ref{tm:sharp_hechler}]
To prove the first part of the theorem, it is enough  to show that $\mc{I}$
has the Hechler property.
However 
\begin{itemize}
\item Hechler proved in \cite{He}, that $\mc{B}$ has the Hechler property,
 \item Bartoszynski and   Kada showed in  \cite{BaKa} that  $\mc{M}$ has the Hechler property,
\item Burke and   Kada proved in  \cite{BuKa} that  $\mc{N}$ has the Hechler property.
\end{itemize}
This proves the first part of the theorem.

Assume now that $Y\subs \pcf(\{\aleph_n:1 \le n<\omega\})$.
Then $A=\pcf(Y)$ has cardinality $<\omega_4$  by the celebrated theorem of Shelah.
Thus $|A|<\min(A)^{+4}$, so we can apply the first part of the present Theorem.
 \end{proof}

\medskip

\noindent{\em Remark.} The problem whether $\mc{N}$ and  $\mc{M}$ have the Hecler property was
raised a preliminary version of the present paper. 

\medskip

\begin{corollary}\label{cor:sharp_hechler}
 If the  ideal $\mc{I}$ has the Hechler property and
 $\cf(\br \aleph_\omega;\omega;,\subs)>\aleph_{{\omega}+1}$
then in some c.c.c generic extension
$\nad (\mc{I})\cap\aleph_{\omega}$ 
is infinite but 
$\aleph_{\omega+1}\notin \nad (\mc{I})$.
\end{corollary}

\begin{proof}[Proof of the corollary]
 If $\max \operatorname{pcf}(\{\aleph_n:1\le
n<{\omega}\})=\cf(\br \aleph_\omega;\omega;,\subs)>\aleph_{{\omega}+1}$ then there is an infinite 
set $X\subs \{\aleph_n:n\in {\omega}\}$
such that  $\pcf(X)=X\cup\{\aleph_{\omega+2}\}$ . Now we can apply theorem \ref{tm:sharp_hechler} for that 
$A=X\cup\{\aleph_{\omega+2}\}$
to obtain the desired extension.
\end{proof}

\begin{proof}[Proof of theorem \ref{tm:sharp_hechler}]
Since $|A|<\min(A)^{+n}$, there is a partition
$F\cup^* Y$ of $A$ such that $F$ is finite, $Y$ is progressive, and 
$\max (F)<\min(Y)$.
 
Let $Q=\<\prod A,\le\>$, where  $f\le f'$ iff
$f({\kappa})\le g({\kappa}) $ for each ${\kappa}\in A$. Then $Q$ is
${\sigma}$-directed because $\aleph_0\notin A$. Since  $\mc{I}$ is Hechler, there is a c.c.c poset $P$
such that in $V^P$  the ideal $\mc{I}$ has a cofinal subset  $\{I_q:q\in Q\}$
which is order-isomorphic to $Q$, i.e. $I_q\subs I_{q'}$ iff $q\le_Q q'$.

We are going to show that the model $V^P$ satisfies our requirement.

\begin{claim}
$A\subs \nad (\mc{I})$.
\end{claim}
\begin{proof}
Fix $\kappa\in A$. For each  ${\alpha}<\kappa$ consider the
function $g_{\alpha}\in \prod A$ defined by the
formula
\begin{displaymath}
g_\alpha(a)=\left\{\begin{array}{ll}
{\alpha}&\text{if $a={\kappa},$}\\
0&\text{otherwise.}
\end{array}\right.
\end{displaymath}
Then $\{g_{\alpha}:{\alpha}<\aleph_n\}$ is $\le$-increasing and unbounded in 
$Q$ so $\{I_{g_{\alpha}}:{\alpha}<\kappa\}$ is increasing
and unbounded in $\<\mc{I},\subs\>$.
Hence $\kappa\in \nad (\mc{I})$.
\end{proof}

\begin{claim}
$\nad(\mc{I})\subs A$.
\end{claim}

\begin{proof}[Proof of the claim]

Assume that $\lambda\in \mf{Reg}\setm A$.
We show show that 
$\lambda \notin \nad (\mc{I})$.

Let 
$\mf J=\{J_{\alpha}:{\alpha}<\lambda\}
\subs \mc{I}$ be increasing.

For each ${\alpha}<\lambda$ pick $g_{\alpha}\in \prod A$
such that 
$J_{\alpha}\subs  I_{g_{\alpha}}$.

Since $\lambda \notin \pcf(A)$,  applying Proposition  \ref{pr:bound} for 
$Y$ and Observation  
\ref{ob:finite} for $F$ we obtain
 $K\in \br \lambda;\lambda;$ and $s\in \prod A$ such that 
$g_\alpha\le s$ for each $\alpha\in K$.

Thus $J_\alpha\subs I_s$ for $\alpha \in I$.
Since the sequence $\mf J=\{J_{\alpha}:{\alpha}<\lambda\}$ is increasing
and $K$ is cofinal in $\lambda$
we have 
\begin{displaymath}
 \cup \{J_{\alpha}:{\alpha}<\lambda\}=
 \cup \{J_{\alpha}:{\alpha}\in K\}\subs I_s.
\end{displaymath}
So the sequence  $\mf J=\{J_{\alpha}:{\alpha}<\lambda\}$
does not witness that $\lambda\in \nad(\mc{I})$.

Since $\mf J$ was arbitrary, we proved the claim.
\end{proof}

The two claims complete the proof of the theorem.
\end{proof}

\section{Restrictions on the additivity spectrum}
\label{sc:gen_pos}

The first theorem we prove here resembles to \cite[Theorem 2.1]{ST}. 

\begin{theorem}\label{tm:gen_pos}
Assume that $\mc{I}\subs \mc{P}(I)$ is a $\sigma$-complete  ideal, $Y\in \mc{I}^+$, and 
$A\subs \nad(\mc{I},Y)$ is countable. Then $\pcf(A)\subs \nad(\mc{I},Y)$.  
\end{theorem}

\begin{proof}
For each $a\in A$ fix an increasing  sequence $\mathfrak
F_a=\{F^a_{\alpha}:{\alpha}<a\} \subs \mc{I}$
such that $\bigcup\mathfrak F_a=Y$.

Let ${\kappa}\in \operatorname{pcf} (A)$. 
Fix an ultrafilter $\mc{U}$ on $A$
such  that $\operatorname{cf}(\prod A/\mc{I})={\kappa}$ and fix an
$\le_\mc{U}$-increasing, $\le_\mc{U}$-cofinal sequence 
$\{g_{\alpha}:{\alpha}<{\kappa}\}\subs \prod A$.
For $g\in \prod A$ let
\begin{displaymath}
U(g)=\bigl\{x\in I: 
\{a\in A:x\in F^a_{g(a)}\}\in \mc{U}\bigr\}.
\end{displaymath}
In the next three claims we show that the sequence
$\{U(g_\alpha):\alpha<\kappa\}$ 
witnesses $\kappa\in \nad(\mc{I},Y)$.

\begin{claim}
$U(g)\in\mc{I}$ for each $g\in \prod A$.
\end{claim}
Indeed,
$U(g)\subset \bigcup\{F^a_{g(a)}:a\in A\}\in\mc{I}$ because 
$\mc{I}$ is $\sigma$-complete.

\begin{claim}\label{lm:inclnull}
If $g_1,g_2\in \prod A$, $g_1\le_\mc{I} g_2$ then 
$U(g_1)\subset U(g_2)$.
\end{claim}

Indeed, fix $x\in I$. Since
\begin{displaymath}
\{a\in A:x\in F^a_{g_2(a)}\}\supset
\{a\in A:x\in F^a_{g_1(a)}\}\cap \{a\in A:g_1(a)\le g_2(a)\}
\end{displaymath}
and  $ \{a\in A:g_1(a)\le g_2(a)\}\in \mc{U}$, we have that 
$\{a\in A:x\in F^a_{g_1(a)}\}\in \mc{U}$ implies
$\{a\in A:x\in F^a_{g_2(a)}\}\in \mc{U}$, i.e., if 
$x\in U(g_1)$ then $x\in U(g_2)$, too.

\begin{claim}
$\bigcup\{U(g_{\alpha}):{\alpha}<{\kappa}\}=Y$.
\end{claim}

Indeed, fix $y\in Y$. For each $a\in A$ choose $g(a)<a$ such that 
$y\in F^a_{g(a)}$.  Then $y\in U(g)$.
Pick  ${\alpha}<{\kappa}$ such that $g\le_\mc{U} g_{\alpha}$.
Then $U(g)\subs U({g_{\alpha}})$ and so $y\in U(g_{\alpha})$.

The three claims together give
that sequence $\<U(g_{\alpha}):{\alpha}<{\kappa}\>\subs \mc{I}$ really witnesses that 
${\kappa}\in \nad(\mc{I},Y)$.
\end{proof}

\begin{corollary}\label{cor:gen_pos}
If $\mc{I}\in \{\mc{B},\mc{N},\mc{M}\}$, $Y\in \mc{I}^+$, and 
$A\subs \nad(\mc{I},Y)$ is countable,  then $\pcf(A)\subs \nad(\mc{I},Y)$.  
\end{corollary}

As we will see in the next two subsection, for the ideals
$\mc{B}$ and $\mc{N}$ we can prove stronger closedness properties.

\subsection{The ideal $\mc{B}$}

If $F\subs \omega^\omega$ and $h\in \omega^\omega$ we write
$F\le^* h$ iff $f\le^* h$ for each $f\in F$.

\begin{theorem}\label{tm:b_closed}
If $A\subs \nad(\mc{B})$ is progressive and $|A|<\mathfrak h$, then $\operatorname{pcf} (A)\subs \nad(\mc{B})$.
\end{theorem}

\begin{proof}
For each $a\in A$ fix an increasing sequence   $\mathfrak  F_a=\{F^a_{\alpha}:{\alpha}<a\}\subs \mc{B}$ with
$\cup\mf F_a\notin \mc{B}$.
We can  assume that the functions in the families  $F^a_{\alpha}$
are all monotone increasing.

Let ${\kappa}\in \operatorname{pcf} (A)$. 
Pick an ultrafilter $\mc{U}$ on $A$
such  that $\operatorname{cf}(\prod A/\mc{U})={\kappa}$ and fix an
$\le_\mc{U}$-increasing, $\le_\mc{U}$-cofinal sequence 
$\{g_{\alpha}:{\alpha}<{\kappa}\}\subs \prod A$.

For $g\in \prod A$ let
\begin{displaymath}
\Bd(g)=\big\{h\in {\omega}^{\omega}:\{a\in A:F^a_{g(a)}\le^* h\}\in
\mc{U}\big\},
\end{displaymath}
and 
\begin{displaymath}
\Ke(g)=\{x\in {\omega}^{\omega}:x\le^* h\text{ for each }h\in\Bd(g)\}.
\end{displaymath}

\begin{claim}\label{lm:incl}
For $g_1,g_2\in \prod A$, if $g_1\le_\mc{U} g_2$ then we have $\Bd(g_1)\supset \Bd(g_2)$
and $\Ke(g_1)\subset \Ke(g_2)$.
\end{claim}

\begin{proof}[Proof of the claim%\ref{lm:incl}
]
For each $h\in {\omega}^{\omega}$,
\begin{displaymath}
\{a\in A:F^a_{g_1(a)}\le^* h\}\supset
\{a\in A:F^a_{g_2(a)}\le^* h\}\cap \{a\in A:g_1(a)\le g_2(a)\}.
\end{displaymath}
Since $ \{a\in A:g_1(a)\le g_2(a)\}\in \mc{U}$, we have that 
$\{a\in A:F^a_{g_2(a)}\le^* h\}\in \mc{U}$ implies
$\{a\in A:F^a_{g_1(a)}\le^* h\}\in \mc{U}$, i.e., if 
$h\in \Bd(g_2)$ then $h\in \Bd(g_1)$, too.

From the relation $\Bd(g_1)\supset \Bd(g_2)$ the inclusion  
$\Ke(g_1)\subset \Ke(g_2)$ is straightforward by the definition of the
operator $\Ke$.
\end{proof}

\begin{claim}\label{lm:bd}
$\Bd(g)\ne \empt$ for each $g\in \prod A$.
\end{claim}

Indeed, for each $a\in A$ let $h_a\in {\omega}^{\omega}$ such that 
$F^a_{g(a)}\le^* h_{\alpha}$. Since $|A|< \mf h\le \mathfrak b$ there is 
$h\in {\omega}^{\omega}$ such that $h_a\le^*h$ for each $a\in A$.
Then $h\in \Bd(g)$. 
%\end{proof}

\smallskip

\begin{claim}
The sequence $\mathfrak F=\<\Ke(g_{\alpha}):\alpha<\kappa\>$  witnesses that 
${\kappa}\in \nad (\mc{B})$.
\end{claim}

By claim \ref{lm:incl}, we have ${\Ke(g_{\alpha})}\subs {\Ke(g_{\beta})}$ for ${\alpha}<{\beta}<\kappa$,
and  each ${\Ke(g_{\alpha})}$ is in $\mc{B} $ by  claim \ref{lm:bd}.

So all we  need is to show that $F=\bigcup\{{\Ke(g_{\alpha})}:{\alpha}<{\kappa}\}\notin\mc{B}$, i.e. 
$F$ is not $\le^*$-bounded.
Let $x\in {\omega}^{\omega}$ be arbitrary. 
We will find $y\in F$ such that $y\not\le^* x$.

For each $a\in A$ let $F^a=\cup\{F^a_\alpha:\alpha<\}$, and put 
\begin{displaymath}
\mc{J}(a)=\{E\subs {\omega}:\exists f\in F^a\
x\restr E <^*f\restr E\}.
\end{displaymath}
Since the functions in $F^a$ are all monotone increasing and 
$F^a$ is unbounded in $\<\omega^\omega,\le^*\>$, for each 
$B\in \br \omega;\omega;$ the family 
$\{f\restriction B:f\in F^a\}$ is unbounded in $\<\omega^B,\le^*\>$,
so $B$ contains some element of $\mc{J}(a)$. In other words,
$\mc{J}(a)$ is  dense in
$\<{\omega}^{\omega},\subset^*\>$.
Since  every $\mc{J}(a)$ is clearly open and   $|A|<\mathfrak h$, 
\begin{displaymath}
\mc{J}={\bigcap}\{\mc{J}(a):a\in A\}
\end{displaymath}
 is also dense in $\<{\omega}^{\omega},\subset^*\>$.
Fix an  arbitrary $E\in \mc{J}$. For each $a\in A$ 
pick $f^a\in F^a$ which witnesses that $E\in \mc{J}(a)$, i.e.
$x\restriction E <^* f^a$.
Choose  $g(a)<a$ with $f^a\in F^a_{g(a)}$. 

Define the function $y\in {\omega}^{\omega}$ as follows:
\begin{displaymath}
y(n)=\left\{\begin{array}{ll}
x(n)+1&\text{if $n\in E$,}\\
0&\text{otherwise.}
\end{array}\right.
\end{displaymath}

Then $y\le^* f^a\in F^a_{g(a)}$ and so if $F^a_{g(a)}\le^* h$ then 
$y\le^* h$. Thus $y\in \Ke(g)$.
Fix ${\alpha}<{\kappa}$ such that $g\le_\mc{U} g_{\alpha}$.
By lemma \ref{lm:incl}, $\Ke (g)\subs \Ke(g_{\alpha})$, hence
$y\in {\Ke(g_{\alpha})}\subs F$ and clearly $y\not\le^* x$
and so $F\not\le^* x$.
Since $x$ was an arbitrary elements of ${\omega}^{\omega}$, we are done.
\end{proof}

\subsection{The ideal $\mc{N}$}

\begin{theorem}\label{tm:n_closed}
If $A\subs \nad(\mc{N})\cap{Reg}$ is countable,
then $\pcf (A)\subs \nad(\mc{N})$.
\end{theorem}

To prove the theorem above we need some preparation.
%We will use the following fact.
Denote $\lambda$ the product measure on $2^\omega$, and
$\lambda_\omega$ the product measure of countable many copies of $\<2^\omega,\lambda\>$.
By \cite[417J]{Fr4} the products of measures are associative.
Since $\omega\times\omega=\omega$, and   $\<2^\omega,\lambda\>$ itself is the product of 
countable many copy of a measure space on 2 elements,     we have the following fact.
\begin{fact}\label{fa:pr}
There is a bijection $f:2^\omega\to (2^\omega)^\omega$ such that 
$\lambda(X)=\lambda_\omega(f[X])$ for each $\lambda$-measurable set $X\subs 2^\omega$. 
So
\begin{equation}\tag{$\dag$}
\label{eq:add}  \nad(\mc{N})=\nad(\mc{N}_\omega), 
\end{equation}
where $\mc{N}_\omega=\{X\subs (2^\omega)^\omega:\lambda_\omega(X)=0\}$.
\end{fact}

Denote $\lambda^*$  the outer measure on $2^\omega$.
Clearly for some $X\subs 2^\omega$ we have $\lambda^*(X)>0$ iff $X\notin \mc{N}$.

As we will see soon, Theorem \ref{tm:n_closed} follows easily from the next result.

\begin{theorem}
\label{lm:merge}
If $A\subs \nad(\mc{N})$ is countable then there is 
$Y\subs \Int$ such that ${\lambda}^*(Y)=1$ and 
$A\subs  \nad(\mc{N},Y)$.
\end{theorem}

\begin{proof}[Proof of theorem \ref{tm:n_closed} from Theorem \ref{lm:merge}]
By Theorem \ref{lm:merge} there is $Y\subs \Int$ such that 
$A\subs \nad(\mc{N}, Y)$ and ${\lambda}^*(Y)=1$.
Now apply theorem \ref{tm:gen_pos} for $Y$ and $A$ to obtain 
$\pcf(A)\subs \nad(\mc{N}, Y)\subs \nad(\mc N)$.
\end{proof}

\begin{proof}[Proof of Theorem \ref{lm:merge}]
First we prove some easy claims.

\begin{claim}\label{f:incre}
If $X\subs \Int$ is measurable, $1>{\lambda}(X)>0$ then there is 
$x\in \Int$ such that ${\lambda}(X\cup (X+x))>{\lambda}(X)$,
where $X+x=\{x'+x:x'\in X\}$.
\end{claim}

\begin{proof}[Proof of the claim]
By the Lebesgue density theorem there are $y,z\in \Int$ and $\varepsilon >0$ such that 
for each $0<{\delta}<\varepsilon$ we have
${\lambda}(X\cap [y-{\delta},y+{\delta}])>{\delta}$ and 
${\lambda}(X\cap [z-{\delta},z+{\delta}])<{\delta}$.
Let $x=z-y$. Then 
${\lambda}((X\cup (X+x))\cap [z-{\delta},z+{\delta}])\ge
{\lambda}(X\cap [y-{\delta},y+{\delta}])>{\delta}>
{\lambda}(X\cap [z-{\delta},z+{\delta}])
$. So ${\lambda}(X\cup (X+x))>{\lambda}(X)$.
\end{proof}

\begin{claim}\label{f:max}
If $X\subs \Int$  is Lebesgue-measurable, ${\lambda}(X)>0$ then there is  a set 
$\{x_n:n<{\omega}\}\subs 2^\omega$
 such that ${\lambda}(\bigcup\{X+x_n:n\in {\omega}\})=1$.
\end{claim}

\begin{proof}[Proof of the claim]
Apply claim \ref{f:incre} until you can increase the measure.
We should stop after  countable many steps.
\end{proof}

\begin{claim}\label{f:max2}
If $Y\subs \Int$,  ${\lambda}^*(Y)>0$ then there are real numbers
$\{x_n:n<{\omega}\}$ such that ${\lambda}^*(\bigcup\{Y+x_n:n\in {\omega}\})=1$.
\end{claim}

\begin{proof}[Proof of the claim]
Fix a Lebesgue measurable set $Y$ such that $X\subs Y$ and for
each measurable set $Z$ with $Z\subs Y\setm X$ we have ${\lambda}(Z)=0$.
Apply claim \ref{f:max} for $Y$: we obtain a set
$\{x_n:n<{\omega}\}\subs \Int$ such that taking 
$Y^*=\bigcup\{Y+x_n:n<{\omega}\}$ we have ${\lambda}(Y^*)=1$.
Let $X^*=\bigcup\{X+x_n:n<{\omega}\}$.
Then ${\lambda}(X^*)=1$. Indeed, if 
$Z\subs Y^*%\setm X^*
$ is measurable with  ${\lambda}(Z)>0$ then 
there is $n$ such that ${\lambda}(Z\cap (Y+x_n))>0$.
Let $T=(Z-x_n)\cap Y$. Then $T\subs Y$ is measurable with ${\lambda}(T)>0$, so 
there is $t \in T\cap X$. Then $t+x_n\in Z\cap X^*$, i.e.  $Z\not\subs  Y^*\setm X^*$.
\end{proof}

\begin{lemma}\label{lm:eto1}
If $0<\lambda^*(X)$  then there is 
${X^*}\subs \Int$ such that ${\lambda}^*({X^*})=1$ and $\nad(\mc{N},{X^*})=\nad(\mc{N}, X)$.
\end{lemma}

\begin{proof}
Fix $\{x_n:n<{\omega}\}\subs \Int$ such that ${\lambda}({X^*})=1$, 
where ${X^*}=\bigcup\{X+x_n:n<{\omega}\}$.
If ${\kappa}\in\nad(\mc{N}, X)$ then there is a sequence
$ \<I_{\nu}:{\nu}<{\kappa}\>\subs \mc{N}$ such that
$\bigcup_{{\zeta}<{\nu}}I_{\zeta}\in \mc{N}$ 
for each ${\nu}<{\kappa}$ and
$\bigcup_{{\zeta}<{\kappa}}I_{\nu}=X$.
Let $J_{\nu}=\bigcup\{I_{\nu}+x_n:n<{\omega}\}$.
Then the sequence $\<J_{\nu}:{\nu}<{\kappa}\>$ witnesses 
${\kappa}\in \nad(\mc{N}, {X^*})$.

If $\<J_{\nu}:{\nu}<{\kappa}\>$ witnesses that ${\kappa}\in \nad(\mc{N}, {X^*})$
then $I_{\nu}=J_{\nu}\cap X$ witnesses that ${\kappa}\in \nad(\mc{N}, X)$.
\end{proof}

Denote $\lambda_\omega^*$ the outer measure on $(2^\omega)^\omega$.

\begin{lemma}\label{lm:product}
If $\{Y_n:n<\omega\}\subs \mc{P}(2^\omega)$ with 
$\lambda^*(Y_n)=1$  then $\lambda_\omega^*(\prod Y_n)=1$.
\end{lemma}

\begin{proof}
Write $Y^*=\prod Y_n$.

Assume on the contrary that there is $Z\subs (2^\omega)^\omega\setm Y^*$
with $\lambda(Z)>0$. Since the measure $\lambda_\omega$ is regular, we can assume that 
$Z$ is compact. 
By induction we pick elements
$y_0\in Y_0,\dots, y_n\in Y_n, \dots$ such that
$\lambda(Z_n)>0$, where
\begin{displaymath}
Z_n=\{z\in (2^\omega)^\omega: \<z_0,\dots, z_{n-1}\>^\frown z\in Z \}. 
\end{displaymath}
Especially $Z_0=Z$.

If $Z_n$ is defined let
\begin{displaymath}
T_n=\{t\in 2^\omega:  \lambda(\{z:\<t\>^\frown z\in Z_n  \})>0 \} 
\end{displaymath}
By Fubini theorem, $\lambda(T_n)>0$, so we can pick $y_n\in T_n\cap Y_n$.

Let $y=\<y_n:n<\omega\>\in \prod Y_n$. Then for each $n\in \omega$ there is $z$ such that 
$y\restriction n^\frown z\in Z$, and so $y\in Z$ because $Z$ is compact. 
\end{proof}

We are ready to conclude the proof of Theorem \ref{lm:merge}.

Enumerate first $A$ as $\{{\kappa}_n:n<{\omega}\}$.
For each $n<{\omega}$ apply lemma \ref{lm:eto1} to get 
$X_n\subs\Int $ such that ${\lambda}^*(X_n)=1$ and 
${\kappa}_n\in \nad(\mc{N},{X_n}$). Let $X^*=\prod_{n\in {\omega}} X_n\subs 
(\Int)^{\omega}$. 
Then ${\lambda}^*(X^*)=1$ and 
$A=\{{\kappa}_n:n<{\omega}\}\subs\nad(\mc{N}_{\omega}, X^*)$.
%, where
%$\mc{N}_{\omega}=\{Z\subs \Int^{\omega}:{\lambda}(Z)=0\}$.
Thus $\pcf(A)\subs \nad(\mc{N}_{\omega}, X^*)\subs \nad(\mc{N}_\omega)$
by Theorem \ref{tm:gen_pos}.
However, $\nad(\mc{N}_\omega)=\nad(\mc{N})$ by (\ref{eq:add})
from Fact \ref{fa:pr}, so we are done. 
\end{proof}

% Finally, the following corollary claims that in some special cases we can give full characterization
% of the additivity spectrum of $\mc B$ and $\mc N$. 

\begin{corollary}\label{corr:count}
Let $\mc I$ be either the ideal $\mc B$ or the ideal $\mc N$. 
Assume that $A$ is a  non-empty set $A$ of uncountable regular cardinals.
If  $A$ is countable, or $\max A\le\ cf(\br \aleph_\omega;\omega;,\subs)$ then 
the following statements are equivalent:
\begin{enumerate}[(1)]
\item $A=\nad(\mc I)$ in some c.c.c extension of the ground model,
\item $A=\pcf(A)$.
\end{enumerate}
\end{corollary}

\begin{proof}
(2) $\Longrightarrow $ (1):  if $A$ is countable then $A$ is progressive. 

If $\sup(A)\le \cf(\br \aleph_\omega;\omega;,\subs)$,
 then we have $A\subs \pcf(\aleph_n:1\le n<\omega)$, and so   $|A|<\omega_4\le \min(A)^{+4}$ by the celebrated theorem of Shelah
\cite{Sh}.
 
So in both case we can apply Theorem \ref{tm:sharp_hechler} to get (1).

\medskip
\noindent
(1) $\Longrightarrow $ (2): By Theorems \ref{tm:b_closed} and \ref{tm:n_closed} we have that
\begin{displaymath}\tag{$\star$}\label{star}
A=\cup\{\pcf(A'):A'\in \br A;\omega;\}. 
\end{displaymath}
If $A$ is countable, (\ref{star}) gives immediately $A=\pcf(A)$.

If $\sup(A)\le \cf(\br \aleph_\omega;\omega;,\subs)$,
 then  $A\subs \pcf(\aleph_n:1\le n<\omega)$, 
so by the Localization Theorem (see \cite[Theorem 6.6.]{AM}) we have 
$\pcf(A)=\cup\{\pcf(A'):A'\in \br A;\omega;\}$. Thus even in this case,
(\ref{star}) gives $A=\pcf(A)$.
\end{proof}

%\section{Problems}

Finally we mention an open question.
We could not prove that if $A\subs \nad(\mc M)$ is countable then $\pcf(A)\subs \nad(\mc M)$
because the following question is open:
\begin{problem}
%\label{lm:merge}
Is it true that if $A\subs \nad( {\mc{M}})$ is countable then  $A\subs \nad({\mc{M},Y})$
for some $Y\notin \mc M$?
\end{problem}

% \begin{problem}
% Is it true that $\nad(\mc B)= \nad(\<\omega^\omega,\le^*\>)$?
% \end{problem}

\end{document}